\documentclass[11pt]{article}

\usepackage{amsthm,amssymb,amsfonts,amsmath}
\usepackage{amscd} 
\usepackage{mathrsfs}
\usepackage[numbers, sort&compress]{natbib}
\usepackage{graphicx}
\usepackage[pagebackref=true]{hyperref}
\usepackage{color}

\providecommand{\R}{}

\renewcommand{\R}{\mathbb{R}}

\providecommand{\ora}[1]{}
\renewcommand{\ora}[1]{\overrightarrow{#1}}

\newtheorem{thm}{Theorem}
\newtheorem{lem}[thm]{Lemma}

\newtheorem{cor}[thm]{Corollary}

\numberwithin{thm}{section}
\numberwithin{equation}{section}

\usepackage{graphicx}
\usepackage{algpseudocode}
\usepackage{amssymb}
\usepackage{geometry}

\begin{document}
\title{A Metric Extension Theorem}
\author{Jacob Canel\\ Department of Mathematics, Pennsylvania State University, \\ \href{mailto:jmc8684@psu.edu}{jmc8684@psu.edu}}
\maketitle
\begin{abstract}
    We show that, given a partially defined metric $d:R\rightarrow\R$ where $R\subset X^2$ is a closed equivalence relation on a compact metrizable space $X$, there is a metric $\hat{d}:X^2\rightarrow \R$ such that $d = \hat{d}|_{R}$
\end{abstract}
\section{Introduction}

We prove an analogue of the Tietze extension theorem for partially defined metrics which have known values on some closed equivalence relation $R\subset X^2$. In particular, we prove necessary and sufficient conditions for the existence of an extension of a partially defined metric on such an $R$. The simplest case of such an extension is a corollary of a result by Hausdorff, frequently called the Hausdorff extension theorem, proven in his 1930 paper \cite{Hausdorff1930} about extending homeomorphisms. A simple proof of the theorem is given by Toruńczyk \cite{Toruńczyk1972}. 

It can be stated as follows:
\begin{thm}[Hausdorff]
    Let $A\subset X$ be a closed subset of a compact metric space $X$, and let $d_A$ be a metric on $A$, then $d_A$ extends to $d$, a metric on $X$.
\end{thm}

This is a subcase of the one of interest to us, under the equivalence relation $a \cong_R b$ iff $a = b$ or $a,b\in A$, with the partially defined metric $d_R$ defined by $d_R(a,b) = d_A(a,b)$ if both are in $A$, and zero if $a = b$.\\

Much of the modern work extending this classical theorem is by Yoshito Ishiki, who has proved interpolation results about families of metric extensions \cite{ishiki2026interpolationmetricsspacesmetrics}, constructed a function extending metrics on the subset $A$ with uniformity on the Hausdorff distance \cite{ishiki2024isometricextensormetrics}, and has extended the result to the case where the space $X$ is $\sigma$-compact and locally compact \cite{Ishiki_2023}, among other related results on ultrametrics \cite{ishiki2021embeddingextensioninterpolationultrametrics} \cite{Ishikifactor}.

Other notable work on related problems, includes the work of Isbell (particularly lemma 1.4) \cite{pjm/1103039456}, which gives an extension result for pseudometrics; Banakh et al. (see theorem 4.1) \cite{article}, which discusses the uniform regularity of an operator which extends such metrics; and Tymchatyn and coauthors \cite{Tymchatyn1}\cite{Tymchatyn2} on variable domains.

The work of this paper concerns a basic result about the topology of metric spaces, which the author believes to be a new result. This is in some respect surprising given how classical the area and theorems used are, but the author was unable to find either the theorem statement or a stronger statement in the literature. The author thanks Omar Mohsen for telling the author about this problem on a long drive from Penn State to Toronto.

\section{The Main Theorem}
\begin{thm}\label{thm:maintheorem}
 If $X$ is a compact metrizeable space and $R\subset X^2$ is a closed equivalence relation, and $d:R\rightarrow\R$ is a continuous family of metrics on the fibers, it may be extended to a metric on the entire space $X$.
\end{thm}

We reserve $X$ for a compact metrizeable space and $R$ for a closed equivalence relation. $X/R$ will be the quotient of $X$ by the relation $R$, and we denote by $[a]$ the equivalence class in $X/R$ of $a$. As $R$ is closed in $X^2$, $X/R$ is also a compact metric space with the metric $r([x],[y]) = dist_m((x,y),R )$,  where $m$ is some metric on $X^2$. We further reserve $d$ as the continuous family of metrics on $R$, by which we mean a continuous map $d:R\rightarrow \R$ such that $d|_{[x]\times [x]}$ is a metric. The proof of the theorem is by a sequence of constructions.

\begin{lem}\label{lem:dominatingmetric}
    Given a metric $r$ on $X/R$. There is a metric $g$ such that $g(x,y)\geq r([x],[y])$ and, in the case $x =_R y$, $g(x,y) \geq d(x,y)$.
\end{lem}
\begin{proof}
    Using the Tietze extension theorem for functions, extend $d$ to a continuous function $\hat{d}$ on $X^2$. We define a pseudometric $h$ such that $\hat{d}$ function will be $1$-Lipschitz with respect to $h$. $h$ is defined by
    \[ h(x,y) = \max_{z\in X} |\hat{d}(x,z) - \hat{d}(y,z)|+ \max_{z\in X}|\hat{d}(z,x)-\hat{d}(z,y)|\]
    We claim that $h$ is a pseudometric as symmetry and nonnegativity are implicit from the definition. Trivially, we have
    \[|\hat{d}(x,z) - \hat{d}(y,z)| \leq |\hat{d}(x,z) - \hat{d}(p,z)|+  |\hat{d}(p,z) - \hat{d}(y,z)|\]
      \[ |\hat{d}(z,x)-\hat{d}(z,y)| \leq  |\hat{d}(z,x)-\hat{d}(z,p)| + |\hat{d}(z,p)-\hat{d}(z,y)|\]
    for any $z$ and $p$, and so we have
    \[h(x,y)\leq h(x,p) + h(y,p)\]
    by taking maxima over $z$.\\
    $h(x,y)$ separates any points within an $R$ equivalence class, as taking $z = x$ gives $h(x,y) \geq |\hat{d}(x,x) - \hat{d}(x,y)| = d(x,y)$ for $y =_R x$. If we then define
    \[g(x,y) = h(x,y) + r([x],[y])\]
    We will get a pseudometric which separates all points in $X$, thus a metric on $X$.
\end{proof}

\begin{cor}\label{cor:metformet}
    $X$ is metrizeable iff $X/R$ is metrizeable, where $R$ carries a continuous family of metrics along its fibers.
\end{cor}
\begin{proof}
    This is immediate from the previous argument, without the assumption that $r$ comes from a metric on $X$.
\end{proof}

We now have a metric large enough to accommodate $d$ beneath it. This is half the battle. We now need to look at ways to decrease metrics so that they will be at least as small as $d$. The intuition here is simple: attach a ``tunnel" of length $r$ from point $a$ and take the new distance allowing for transit through the tunnel.

\begin{lem}\label{lem:metricbridge}
    Let $g(x,y)$ be a metric on a space $X$ and let $0< r < g(a,b)$, the function 
    $$f(x,y) =  \min(g(x,y), g(x,a) + r + g(b,y), g(x,b) + r + g(a,y))$$
    is a metric on $X$
\end{lem}
\begin{proof}
    Symmetry and positivity are once again clear. Furthermore this definition is invariant under interchange of $a$ and $b$.  Suppose $f(x,y) > f(x,z) + f(z,y)$ for some $z$. We must have
    \[ f(x,z) + f(z,y) <  g(x,z) + g(z,y)\]
    or else the triangle inequality for $g$ is violated.
    If we have
    \[f(x,z) \neq g(x,z) \text{ and } f(y,z) \neq g(x,z)\]
    then WLOG (by selecting $b$ to be the point closer to $z$)
    \[f(x,z) = g(x,a) + r + g(b,z) \text{ and } f(y,z) = g(y,a) + r + g(b,z)\]
    and $g(x,y) \geq f(x,y) > g(x,a) + g(a,y) $, which violates the triangle inequality. Therefore WLOG $f(x,z) = g(x,z)$ and $f(y,z) = g(y,a) + r + g(b,z) $. We must have that $g(y,a) \leq g(y,b)$, as otherwise $f(y,z) > g(y,z) $, which is impossible. Thus,
    \[f(x,y) \leq g(y,a) + r + g(b,x) \leq g(y,a) + r + g(b,z) + g(z,x) \]
    and consolidating, we obtain
    \[f(x,y) \leq f(y,z) + f(z,x)\]
    which is the triangle inequality for $f$.
\end{proof}

\begin{proof}[Proof of theorem \ref{thm:maintheorem}]
   Let $A$ be the set of all metrics on $X$ for which $d_i(a,y) + d_i(x,b) \geq |d(a,b) - d(x,y) | $, dominate $r([x],[y])$ pointwise everywhere, and are dominated by $g(x,y)$ pointwise everywhere. Lemma \ref{lem:dominatingmetric} shows that this set is nonempty. Order this set by pointwise domination. Suppose we have a descending chain indexed by a linear order $(I,\leq_I)$, i.e. $d_i \leq d_j$ pointwise if $i \leq j$. Define 
   \[d_I(x,y) = \inf_{i\in I}d_i(x,y)\]
   again, symmetry and non-negativity are clear. The triangle inequality follows by splitting the infimum
   \[\inf_{i\in I}d_i(x,y) \leq \inf_{i\in I}d_i(x,z) + \inf_{i\in I}d_i(z,y)\]
   So $d_I$ is a pseudometric.
   \[d_i(a,y) + d_i(x,b) \geq |d(a,b) - d(x,y) |\]
   by assumption, and if $[x] = [y]$
   \[d_i(x,y) = d_i(x,y) + d_i(y,y) \geq |d(x,y) - d(y,y) | \geq d(x,y) \] 
   and if not then $d_i(x,y) \geq r([x],[y])$ by assumption. Taking infima yields that $d_I$ separates points and thus is a metric. Thus $d_I$ is an element of $A$. Zorn's lemma guarantees that $A$ must contain a minimal element. Call one such minimal element $\bar{d}$.\\
   Assume, in hope of a contradiction, that
   \[\bar{d}(a,b) > d(a,b)\]
   for some $a,b$ for which $d$ is defined. The function \[k(x,y) := \min(\bar{d}(x,y),\bar{d}(x,a) + d(a,b) + \bar{d}(b,y), \bar{d}(x,b) + d(b,a) + \bar{d}(a,y)  )\]
   is a metric by lemma \ref{lem:metricbridge}. Furthermore, WLOG we either have
   \[k(x,y) \geq \bar{d}(x,a) + d(a,b) + \bar{d}(b,y) \geq r([x],[a]) + r([b],[y]) \geq r([x],[y])\]
   or 
   \[k(x,y) = \bar{d}(x,y) \geq r([x],[y])\]
   so $k(x,y)$ dominates $r([x],[y])$.\\
   Furthermore
   \[ k(x,y) = \bar{d}(x,a) + d(a,b) + \bar{d}(b,y) \geq d(x,y) \]
   And so $k$ is an element of $A$, but is strictly less than $\bar{d}$, contradicting minimality. Therefore we must have that $\bar{d}(x,y) = d(x,y)$ for all $(x,y)\in R$.

\end{proof}
 \bibliographystyle{plain}
 \bibliography{bibo}
\end{document}